\documentclass{amsproc}

\usepackage{amssymb}
\usepackage{amsfonts}
\usepackage[latin1]{inputenc}

\newtheorem{theorem}{Theorem}[section]
\newtheorem{lemma}[theorem]{Lemma}
\newtheorem{proposition}[theorem]{Proposition}

\theoremstyle{definition}
\newtheorem{definition}[theorem]{Definition}

\theoremstyle{remark}
\newtheorem{remark}[theorem]{Remark}

\numberwithin{equation}{section}

\begin{document}

\title{C*-ALGEBRAS ASSOCIATED WITH ITERATED FUNCTION SYSTEMS}

\author{Gilles G. de Castro}
\address{Instituto de Matemática, Universidade Federal do Rio Grande do Sul, Av. Bento Gonçalves, 9500, 91509-900 Porto Alegre, RS - Brazil}
\curraddr{Département de Mathématiques, Université d'Orléans,  B.P. 6759, 45067 Orléans cedex 2, France}
\email{gillescastro@gmail.com}
\thanks{Partially supported by CAPES}


\date{November 3, 2008}

\begin{abstract}
We review Kajiwara and Watatani's construction of a C*-algebra from an iterated function system (IFS). If the IFS satisfies the finite branch condition or the open set condition, we build an injective homomorphism from Kajiwara-Watatani algebras to the Cuntz algebra, which can be thought as the algebra of the lifted system, and we give the description of its image. Finally, if the IFS admits a left inverse we show that the Kajiwara-Watatani algebra is isomorphic to an Exel's crossed product.
\end{abstract}

\maketitle

\section{Introduction}

In \cite{KW1}, Kajiwara and Watatani defined a C*-algebra defined from an iterated
function system (IFS). Although their paper was entitled C*-algebras associated
with self-similar sets, they gave examples of different iterated function
systems which give rise to the same self-similar set but which associated
algebras are not isomorphic. So their algebra depends not only on the
self-similar set but on the dynamics of the iterated function system.

If the IFS satisfies the strong separation condition, then the system can be interpreted as the inverses branches of a local homeomorphism. For an arbitrary IFS we can lift it to a new one that satisfies the strong separation condition \cite{B1}. Ionescu and Muhly suggested in \cite{IM} the construction of a C*-algebra from an IFS by lifting it and using Renault-Deaconu construction \cite{D1}, \cite{R1} of a groupoid C*-algebra from a local homeomorphism. As we will see this local homeomorphism we find is topologically conjugate to the left shift on $\{1,\ldots,d\}^{\mathbb{N}}$ and the algebra we find is the Cuntz algebra $\mathcal{O}_d$.

From the relations between an IFS and its lifted system, we will build a natural homomorphism from the Kajiwara-Watatani algebra to the Cuntz algebra and thus connect Kajiwara anda Watatani's construction with Ionescu and Muhly's suggestion. We show that this homomorphism is injective and show that its image is generated by the algebra of the self-similar set associated to the IFS and an isometry $S$ similar to a crossed product description.

It may happen that the IFS admits a left inverse which is not necessarily a local homeomorphism. In this case Renault-Deaconu's construction no longer works and we have different approaches to build a C*-algebra. In this paper, we show that under some assumptions, the algebra considered by Kajiwara and Watatani can be seen as an Exel's crossed product \cite{E1}.

\section{Iterated function systems}

In this section, we review some of the basic theory of iterated function systems and self-similar sets (see for instance \cite{B1}, \cite{Ed1} and \cite{F1}). Fix $(X,\rho )$ be a compact metric space.

\begin{definition}
We say that a function $\gamma :X\rightarrow X$ is

\begin{itemize}
\item a contraction if $\exists c\in (0,1)$ such that $\rho (\gamma
(x),\gamma (y))\leq c\rho (x,y)$;

\item a proper contraction if $\exists c_{1},c_{2}\in (0,1)$ such that $%
c_{1}\rho (x,y)\leq \rho (\gamma (x),\gamma (y))\leq c_{2}\rho (x,y)$;

\item a similarity if $\exists c>0$ such that $\rho (\gamma (x),\gamma
(y))=c\rho (x,y)$.
\end{itemize}
\end{definition}

\begin{definition}
An iterated function system (IFS) over $X$ is a finite set of continuous
functions $\left\{ \gamma_{i}:X\rightarrow X\right\} _{i=1}^{d}$%
. We say that the IFS\ is hyperbolic if all functions are contractions.
\end{definition}

Throughout this paper we will always assume that the system is hyperbolic
unless stated otherwise.

\begin{proposition}
Given an IFS $\left\{ \gamma_{i}\right\} _{i=1}^{d}$, there is
a unique compact nonempty subset $K\ $of $X$ such that%
\begin{equation}
K=\cup _{i=1}^{d}\gamma _{i}(K).  \label{eq:InvSet}
\end{equation}%
We will call this set the attractor of the $IFS$ and say it is self-similar.
\end{proposition}

Note that because of (\ref{eq:InvSet}) the attractor is invariant by all $%
\gamma _{i}$ and we can restrict the IFS to its attractor. From now on, we assume that $X=K$.

\begin{definition}
We say that an IFS $\left\{ \gamma_{i}\right\} _{i=1}^{d}$
satisfies:

\begin{itemize}
\item the strong separation condition if the union in (\ref{eq:InvSet}) is a
disjoint union;

\item the open set condition if $\exists U\subseteq K$ open and dense such
that%
\[
U\subseteq \cup _{i=1}^{d}\gamma _{i}(U). 
\]
\end{itemize}
\end{definition}

Let's denote $\Omega =\left\{ 1,\ldots ,d\right\} ^{\mathbb{N}}$ with the
product topology, $\sigma :\Omega \rightarrow \Omega $ the left shift and $%
\sigma _{i}:\Omega \rightarrow \Omega $ the function given by%
\[
\sigma _{i}(i_{0},i_{1,}\ldots )=(i,i_{0},i_{1,}\ldots ) 
\]%
where $i\in \left\{ 1,\ldots ,d\right\} $.

\begin{proposition}
\label{prop:CodeMap}Let $\left\{ \gamma_{i}\right\}
_{i=1}^{d}$ be an IFS and $K\ $its attractor then there is a continuous
surjection $F:\Omega \rightarrow K$ such that $F\circ \sigma _{i}=\gamma
_{i}\circ F$. This map is given by the formula%
\[
F(i_{0},i_{1,}\ldots )=\lim_{n\rightarrow \infty }\gamma _{i_{0}}\circ
\cdots \circ \gamma _{i_{n}}(x)
\]%
for an arbitrary $x\in K$. If the IFS satisfies the strong separation
condition $F$ is a homeomorphism.
\end{proposition}

\begin{remark}
Note that under the strong separation condition, we can define the
function $\gamma =F\circ \sigma \circ F^{-1}$ and in this case the functions 
$\gamma _{i}$ are exactly the inverse branches of $\gamma $. Moreover $F$ gives
us a topological conjugacy between $\gamma$ and the shift $\sigma$.
\end{remark}

For an arbitrary IFS $\left\{ \gamma_{i}\right\}_{i=1}^{d}$, 
we can always build a new one that satisfies the strong
separation condition and which share some properties with the original one
\cite{B1}. We define $\widetilde{X}=K\times \Omega $ and define the
functions $\widetilde{\gamma }_{i}:\widetilde{X}\rightarrow \widetilde{X}$
by $\widetilde{\gamma }_{i}(x,\omega )=(\gamma _{i}(x),\sigma _{i}(\omega ))$.
Let $\widetilde{K}=\{(x,\omega)\in X\times\Omega|F(\omega)=x\}$ then
$$\widetilde{K}=\cup _{i=1}^{d}\widetilde{\gamma _{i}}(\widetilde{K}).$$
And it's easily checked that $\{ \widetilde{\gamma }_{i}:\widetilde{K} %
\to\widetilde{K}\} _{i=1}^{d}$ satisfy the strong separation condition.

\begin{definition}
The IFS $\{ \widetilde{\gamma }_{i}\} _{i=1}^{d}$ as above is called the lifted
system of $\left\{ \gamma_{i}\right\}_{i=1}^{d}$.
\end{definition}

\section{C*-algebras associated with an IFS}

We start this section by giving a description of the Cuntz algebra as a groupoid C*-algebra which will be useful in some proofs. We review some of the key elements of Cuntz-Pimsner algebras which will be used to build Kajiwara-Watatani algebras. We build the homomorphism from Kajiwara-Watatani algebras to the Cuntz algebra from a very natural covariant representation. We give some basic properties of this homomorphism. Finally, in the last subsection, we compare the Kajiwara-Watatani algebra to a crossed product construction.

\subsection{Cuntz algebras}\label{sec:CuntzAlg}

\begin{definition}\cite{C1}
For $d\in\mathbb{N}\backslash\{0\}$, the Cuntz algebra $\mathcal{O}_d$ is the C*-algebra generated by $d$ isometries
satisfying the relation $\sum_{i=1}^d{S_iS_i^*}=1$.
\end{definition}

For the sake of some proofs, we review the construction of the Cuntz algebra
as a groupoid C*-algebra \cite{D1}, \cite{R1}. Let
$$G=\{(\omega,m-n,\tau)\in\Omega\times\mathbb{Z}\times\Omega:m,n\in\mathbb{N};
\sigma^m(\omega)=\sigma^n(\tau)\}$$
with the product and the inverse given by
$$(\omega,m-n,\tau)(\tau,k-l,\nu)=(\omega,(m+k)-(n+l),\nu)$$
$$(\omega,m-n,\tau)^{-1}=(\tau,n-m,\omega).$$

We give a basis for the topology on $G$ by the sets
$$B(U,V,m,n):=\{(\omega,m-n,\tau)\in G:\omega\in U,\tau\in V\}$$
where $m,n\in\mathbb{N}$ and $U$ and $V$ are open subsets of $\Omega$
such that $\sigma^m|_U$, $\sigma^n|_V$ are homeomorphisms with $\sigma^m(U)=
\sigma^n(V)$. With this topology $G$ is étale and so admits a Haar system by the 
counting measures.

The multiplication in $C_c(G)$ is given by
$$(p*q)(\omega,m-n,\tau)=\sum{p(\omega,k-l,\nu)q(\nu,(n+k)-(m+l),\tau)}$$
where the sum is taken over all $k,l\in\mathbb{N}$ and $\nu\in\Omega$ such that $\sigma^k(\omega)=\sigma^l(\nu)$ and $\sigma^{n+k}(\nu)=\sigma^{m+l}(\tau)$; and the involution by
$$p^{\ast}(\omega,m-n,\tau)=\overline{p(\tau,n-m,\omega)}$$
for $p,q\in C_c(G)$.

We refer to \cite{R1} for the construction of a norm in $C_c(G)$. For us, it suffices
to know that there exists a norm in $C_c(G)$ such that its completion with respect
to this norm is $\mathcal{O}_d$.

Finally, given $h\in C(\Omega)$ we define a function $h\in C_c(G)$ by
$$h(\omega,m-n,\tau)=[m=n][\omega=\tau]h(\omega)$$
where $[\cdot]$ is the boolean function that gives $1$ if its argument is true and $0$ otherwise.

We also define a function $S\in C_c(G)$ by
$$S(\omega,m-n,\tau)=[m-n=1][\sigma(\omega)=\tau].$$
And we note that if $\chi_{\overline{i}}$ is the characteristic function of the
cylinder $\overline{i}:=\{\omega\in\Omega:\omega_0=i\}$ 
then $S_i=d^{1/2}\chi_{\overline{i}}*S$ for $i\in\{1,\ldots,d\}$ are $d$ isometries
that satisfies the Cuntz relation and generates $C^{\ast}(G)$.
	
\subsection{Cuntz-Pimsner algebras}
	
	We briefly recall the key elements for the construction of Cuntz-Pimsner algebras (\cite{K1}, \cite{P1}) that will be used throughout the paper. For that fix $A$ a C*-algebra.
	\begin{definition}
		A \emph{(right) Hilbert C*-module over $A$} is a (right-)$A$-module $E$ with a sesquilinear map $\left<\ ,\ \right>:E\times E\to A$ such that:
		\begin{itemize}
			\item[(i)] $\left<\xi,\eta a\right>=\left<\xi,\eta\right>a$;
			\item[(ii)] $(\left<\xi,\eta\right>)^{\ast}=\left<\eta,\xi\right>$;
			\item[(iii)] $\left<\xi,\xi\right>\geq 0$;
			\item[(iv)] $E$ is complete with respect to the norm $||\xi||_2=||\left<\xi,\xi\right>||^{1/2}$
		\end{itemize}
		for $a\in A$ and $\xi,\eta\in E$.
		We say that $E$ is \emph{full} if $\left<E,E\right>$ is dense in $A$.
	\end{definition}
		Let $E$ be a Hilbert C*-module and denote by $\mathcal{L}(E)$ the space of adjointable operators in $E$. We note that $\mathcal{L}(E)$ is a C*-algebra.	
	For $\xi,\eta\in E$ we define an operator $\theta_{\xi,\eta}:E\to E$ by $\theta_{\xi,\eta}(\zeta)=\xi\left<\eta,\zeta\right>$. This is an adjointable operator and we denote by $\mathcal{K}(E)$ the closed subspace of $\mathcal{L}(E)$ generated by all $\theta_{\xi,\eta}$.
	\begin{definition}
		A \emph{C*-correspondence} over $A$ is a Hilbert C*-module $E$ together with a C*-homomorphism $\phi:A\to\mathcal{L}(E)$.
	\end{definition}
	Let ($E$,$\phi$) be a C*-correspondence over $A$ and for simplicity suppose that $\phi$ is faithful. We denote by $J_E$ the ideal $\phi^{-1}(\mathcal{K}(E))$.
	\begin{definition}\label{def:PimsnerCovRep}
		A pair $(\iota,\psi)$ of maps $\iota:A\to B$, $\psi:E\to B$, where $B$ is a C*-algebra and $\iota$ a C*-homomorphism, is said to be a \emph{covariant representation of $E$} if:
		\begin{itemize}
			\item[(i)] $\psi(\phi(a)\xi b)=\iota(a)\psi(\xi)\iota(b)$;
			\item[(ii)] $\psi(\xi)^{\ast}\psi(\eta)=\iota(\left<\xi,\eta\right>)$;
			\item[(iii)] $(\psi,\iota)^{(1)}(\phi(c))=\iota(c)$ where the function $(\psi,\iota)^{(1)}:\mathcal{K}(E)\to B$ is given by $(\psi,\iota)^{(1)}(\theta_{\xi,\eta})=\psi(\xi)\psi(\eta)^{\ast}$,
		\end{itemize}
		for $a,b\in A$, $\xi,\eta\in E$ and $c\in J_E$.
	\end{definition}
	For a C*-correspondence ($E$,$\phi$), there exists an algebra $\mathcal{O}(E)$ and a covariant representation $(k_A,k_E)$ that is universal, in the sense that if $(\iota,\psi)$ is a covariant representation of $E$ in a C*-algebra $B$, there is a unique C*-homomorphism $\iota\times\psi:\mathcal{O}(E)\to B$ such that $\iota=(\iota\times\psi)\circ k_A$ and $\psi=(\iota\times\psi)\circ k_E$
	\begin{definition}
		The algebra $\mathcal{O}(E)$ is called the \emph{Cuntz-Pimsner algebra} of $E$.
	\end{definition}

\subsection{Kajiwara-Watatani algebras}

Let $\Gamma=\left\{ \gamma _{i}\right\} _{i=1}^{d}$ be an iterated function system
and $K$ its attractor. We recall the C*-correspondence defined in \cite{KW1}. We let $A=C(K)$, 
$E=C(\mathcal{G})$ where%
$$\mathcal{G}=\cup _{i=1}^{d}\mathcal{G}_i$$
with
$$\mathcal{G}_i\left\{ (x,y)\in K\times K:x=\gamma _{i}\right\}$$
being the cographs in the terminology of \cite{KW1}. The structure of C*-correspondence is given by%
\[
(\phi(a)\xi b)(x,y)=a(x)\xi(x,y)b(y) 
\]%
and%
\[
\left\langle \xi,\eta\right\rangle _{A}(y)=\sum_{i=1}^{d}\overline{\xi(\gamma
_{i}(y),y)}\eta(\gamma _{i}(y),y) 
\]%
for $a,b\in A$ and $\xi,\eta\in E$.

	\begin{proposition}[\cite{KW1}]\label{prop:BasicResultsCorr}
		$(E=C(\mathcal{G}),\phi)$ is a full C*-correspondence over $A=C(K)$ and $\phi:A\to\mathcal{L}(E)$ is faithful and unital. Moreover, the Hilbert module norm is equivalent to the sup norm in $C(\mathcal{G})$.
	\end{proposition}

\begin{definition}
		The Kajiwara-Watatani algebra $\mathcal{O}_{\Gamma}$ associated to $\Gamma$ is the Cuntz-Pimsner algebra associated to the C*-correspondence defined above.
\end{definition}

Regarding $\mathcal{O}_d$ as $C^*(G)$ as in subsection \ref{sec:CuntzAlg} and recalling the code map $F$ given in proposition \ref{prop:CodeMap}, we define $\iota :A\rightarrow \mathcal{O}_d$ by 
\begin{equation}
\iota (a)(\omega,m-n,\tau)=[m=n][\omega=\tau]a(F(\omega))  \label{eq:iota}
\end{equation}
and $\psi :E\rightarrow \mathcal{O}_d$ by
\begin{equation}
\psi (\xi)(\omega,m-n,\tau)=[m-n=1][\sigma(\omega)=\tau]\xi(F(\omega),F(\tau))  \label{eq:psi}
\end{equation}

Note that if $\sigma(\omega)=\tau$, then $\sigma_{\omega_0}(\tau)=\omega$ and $F(\omega)=\gamma_{\omega_0}(F(\tau))$ so that $(F(\omega),F(\tau))\in \mathcal{G}$ and $\psi$ is well defined.

Before showing that this give us a Cuntz-Pimsner covariant representation,
let us recall some definitions and results from \cite{KW1} and \cite{KW2}.

\begin{definition}
Let $\Gamma=\{\gamma _{1},...,\gamma _{d}\}$ be an IFS, we define the following sets 
	$$B(\gamma _{1},...,\gamma _{d}):=\{x\in K| \exists y\in K\ \exists i\neq j:x=\gamma_i(y)=\gamma_j(y)\};$$
	$$C(\gamma _{1},...,\gamma _{d}):=\{y\in K| \exists i\neq j:\gamma_i(y)=\gamma_j(y)\}.$$
  $$I(x):=\{i\in \{1,...,d\};\exists y\in K:x=\gamma _{i}(y)\}.$$
  We call the points of $B(\Gamma)$ branched points and the points of $C(\Gamma)$ branched values. And we say that $\Gamma$ satisfies the finite branch condition if $C(\Gamma)$ is finite.
\end{definition}

Then $B(\gamma _{1},...,\gamma _{d})$ is a closed set, because 
\[
B(\gamma _{1},...,\gamma _{d})=\cup _{i\neq j}\{x\in \gamma _{i}(K)\cap \gamma _{j}(K);\gamma
_{i}^{-1}(x)=\gamma _{j}^{-1}(x)\}
\]
and each of the union is clearly closed.

\begin{lemma}\label{lemma:NiceOpenSet}
In the above situation, if $x\in K\backslash B(\gamma _{1},...,\gamma _{d})$%
, then there exists an open neighborhood $U_{x}$ of $x$ satisfying the
following:

\begin{itemize}
\item[(i)] $U_{x}\cap B=\emptyset ;$

\item[(ii)] If $i\in I(x)$, then $\gamma _{j}(\gamma _{i}^{-1}(U_{x}))\cap
U_{x}=\emptyset $ for $j\neq i$;

\item[(iii)] If $i\notin I(x)$, then $U_{x}\cap \gamma _{i}(K)=\emptyset$.
\end{itemize}
\end{lemma}

\begin{lemma}\label{lemma:FinBranchIdeal}
If $\Gamma$ satisfies the finite branch condition or the open set condition then $J_E=\{a\in A=C(K);a~\mathrm{vanishes~on~}B(\gamma
_{1},...,\gamma _{d})\}$ where $J_E=\phi^{-1}(\mathcal{K}(E))$ as in the previous subsection.
\end{lemma}

\begin{remark}\label{rem:CompactInvIm}
In the following proof, we will need an explicit description of $\phi(a)$ for certain elements in $J_E$. We do as in \cite{KW1}. Let $B=B(\gamma _{1},...,\gamma _{d})$ and take $a\in A$ such that $Y:=\mathrm{supp}(a)\subseteq K\backslash B$. Clearly $a\in J_E$.

For each $x\in Y$ choose an open neighborhood $U_x$ as in lemma \ref{lemma:NiceOpenSet}. Since $Y$ is compact, there exists a finite set $\{x_1,\ldots,x_m\}$ such that $Y\subseteq \cup_{k=1}^m U_{x_k}$. Let $U_k=U_{x_k}$ for $k=1,\ldots,m$ and $U_{m+1}=K\backslash Y$, then $\{U_k\}_{k=1}^{m+1}$ is an open cover of $K$. Let $\{\varphi_k\}_{k=1}^{m+1}\subseteq C(K)$ be a partition of unity subordinate to this open cover. Define $\xi_k,\eta_k\in C(\mathcal{G})$ by $\xi_k(x,y)=a(x)\sqrt{\varphi_k(x)}$ and $\eta_k(x,y)=\sqrt{\varphi_k(x)}$ then $\phi(a)=\sum_{k=1}^{m}{\theta_{\xi_k,\eta_k}}$ (the summation goes to $m$ only because $\xi_{m+1}=0$).
\end{remark}

\begin{remark}
Because of the lemma \ref{lemma:FinBranchIdeal}, our results will need that the IFS satisfies the finite branch condition or the open set condition, but we note that these conditions are independent.
\end{remark}

\begin{proposition}
\label{PropCovRep}If the IFS $\Gamma$ satisfies the finite branch condition or the open set condition, then the pair $(\iota ,\psi )$ defined by equations (\ref{eq:iota}) and (\ref{eq:psi}) is a Cuntz-Pimsner
covariant representation of $(A,E)$ in $\mathcal{O}_d$.
\end{proposition}

\begin{proof}
Most calculations are very similar so we only show some of them. Let 
$a\in A$ and $\xi\in E$. We have%
\begin{equation}
\psi (a\xi)(\omega,m-n,\tau)=[m-n=1][\sigma(\omega) =\tau]a(F(\omega))\xi(F(\omega),F(\tau))  \label{CovRep01}
\end{equation}%
On the other hand%
$$(\iota (a)\ast \psi (\xi))(\omega,m-n,\tau)=\sum \iota (a)(\omega,k-l,\nu)\psi (\xi)(\nu,(m+l)-(n+k),\tau)=$$
\begin{equation}
a(F(\omega))\psi (\xi)(\omega,m-n,\tau)
\label{CovRep02}
\end{equation}%
where the second equality is true due to the fact that $\iota (a)$ is zero
unless $k=l$ and $\omega=\nu$. We can easily see then that (\ref{CovRep01}) and (\ref{CovRep02}) coincide.

For the $A$-valued scalar product, let $\xi,\eta\in E$. Then%
\[
\iota (\left\langle \xi,\eta\right\rangle _{A})(\omega,m-n,\tau%
)=[m=n][\omega=\tau]\left\langle \xi,\eta\right\rangle _{A}(F(\omega))= 
\]%
\begin{equation}\label{eq:CovRep03}
[m=n][\omega=\tau]\sum_{i=1}^{d}\overline{\xi(\gamma
_{i}(F(\omega)),F(\omega))}\eta(\gamma _{i}(F(\omega)),F(\omega)) 
\end{equation}
and on the other hand%
\[
(\psi (\xi)^{\ast }\ast \psi (\eta))(\omega,m-n,\tau)=\sum \psi
(\xi)^{\ast }(\omega,k-l,\nu)\psi (g)(\nu,(m+l)-(n+k),\eta)= 
\]%
\[
\sum \overline{\psi (\xi)(\nu,l-k,\omega)}\psi (g)(\nu,(m+l)-(n+k),\tau)= 
\]%
\begin{equation}\label{eq:CovRep04}
[m=n][\omega=\tau]\sum_{\sigma(\nu)=\omega}\overline{\xi(F(\nu),F(\omega))}\eta(F(\nu),F(\omega)) 
\end{equation}
and we note that $\sigma(\nu)=\omega$ iff $\nu=\sigma _{i}(\omega)$ for some $i=1,\ldots ,d$ and in this case $F(\nu)=\gamma _{i}(F(\omega))$. It follows that we can rewrite (\ref{eq:CovRep04}) as (\ref{eq:CovRep03}).

Finally, we have to show that $(\psi ,\iota )^{(1)}(\phi (a))=\iota (a)$ for 
$a\in J_E$ where $J_E=\left\{ a\in A:a|_{B(\gamma _{1},\ldots ,\gamma
_{d})}=0\right\} $ by lemma \ref{lemma:FinBranchIdeal}. We take $a\in J_E$ such that $Y:=\mathrm{supp}(a)\subseteq K\backslash B$ and $\xi
_k$ and $\eta _k$ as in remark \ref{rem:CompactInvIm}, then%
\[
(\psi ,\iota )^{(1)}(\phi (a))(\omega,m-n,\tau)=\sum_k(\psi (\xi _k)\ast \psi (\eta _k)^{\ast })(\omega%
,m-n,\tau)= 
\]%
\[
\sum_k\sum \psi (\xi _k)(\omega,k-l,\nu%
)\psi (\eta _k)^{\ast }(\nu,(m+l)-(n+k),\tau)= 
\]%
\[
[m=n][\sigma(\omega)=\sigma(\tau)]\sum_k\xi _k(F(\omega),F(%
\sigma(\omega)))\eta _k(F(\tau),F(\sigma(\omega)))= 
\]%
\begin{equation}
[m=n][\sigma(\omega)=\sigma(\tau)]\sum_ka(F(\omega))\sqrt{\varphi
_k(F(\omega))}\sqrt{\varphi _k(F(\tau))}.  \label{eq:CovRep05}
\end{equation}
Note that $F(\omega)=\gamma _{\omega_{0}}(F(\sigma(\omega)))$ and $\sigma(\omega)=%
\sigma(\tau)$ implies that $F(\tau)=\gamma _{\tau_{0}}(F(\sigma(\omega)))$ where
$\omega_{0},\tau_{0}$ are the coordinates zero of $\omega$ and $\tau$ respectively. Now if $F(\omega)\in
U_{x_k}$ then $\omega_{0}\in I(x_k)$ because of property (iii) of
lemma \ref{lemma:NiceOpenSet} and $F(\sigma(\omega)\in \gamma
_{\omega_{0}}^{-1}(U_{x_k})$. We have that $F(\tau)\in
\gamma _{\tau_{0}}(\gamma _{\omega_{0}}^{-1}(U_{x_k}))$ and if $\omega_{0}\neq
\tau_{0}$ then by property (ii) of lemma \ref{lemma:NiceOpenSet}, we have $F(\tau%
)\notin U_{x_k}$. Since the support of $\varphi _k$ is
contained in $U_{x_k}$ and if $\omega_{0}=\tau_{0}$ then $\omega=%
\tau$, we have from (\ref{eq:CovRep05}) that%
\[
(\psi ,\iota )^{(1)}(\phi (a))(\omega,m-n,\tau)=[m=n][%
\omega=\tau]a(F(\omega))= 
\]%
\[
\iota (a)(\omega,m-n,\tau). 
\]%
As the elements $a\in C(K)$ such that $\mathrm{supp}(a)\subseteq K\backslash B$ are dense in $J_E$, the equality $(\psi ,\iota )^{(1)}(\phi(a))=\iota (a)$ holds for an arbitrary $a\in J_E$.
\end{proof}

\begin{lemma}[\cite{FMR}]\label{lemma:GaugeInvariance}
Suppose that $(\psi ,\iota )$ is an isometric covariant representation of $E$
into a C*-algebra $B$. Then $\psi \times \iota $ is faithful if and only if $%
\iota $ is faithful and there is a (strongly continuous) action $\beta :%
\mathbb{T}\rightarrow \mathrm{Aut}(B)$ such that $\beta _{z}\circ \iota
=\iota $ and $\beta _{z}\circ \psi =z\psi $ for all $z\in \mathbb{T}$.
\end{lemma}

\begin{proposition}
If the IFS $\Gamma$ satisfies the finite branch condition or the open set condition, then the homomorphism $\psi \times \iota $ given by the covariant representation
defined by (\ref{eq:iota}) and (\ref{eq:psi}) is faithful.
\end{proposition}

\begin{proof}
Given $a\in C(K)$,%
\[
(\iota (a)^{\ast }\ast \iota (a))(\omega,m-n,\tau)=[m=n][\omega=\tau]|a(F(\omega))|^{2} 
\]%
and as $F$ is surjective, we have that $\iota $ is faithful. Let $\beta :\mathbb{T}%
\rightarrow \mathrm{Aut}(C^{\ast }(G))$ be the gauge action given by%
\[
\beta _{z}(f)(\omega,m-n,\tau)=z^{m-n}f(\omega,m-n,\tau) 
\]%
then $\beta _{z}(\iota (a))=\iota (a)$ because $\iota (a)$ is zero for $%
m\neq n$; and for $\xi\in E$, $\beta _{z}(\psi (\xi))=z\psi (\xi)$ because $\psi (\xi)$ is zero
for $m-n\neq 1$.
\end{proof}

We conclude with this proposition that $\mathcal{O}_{\Gamma }$ is a
subalgebra of $\mathcal{O}_{d}$. 

\begin{remark}
As $\mathcal{G}$ is a closed, and therefore compact, subset of $K\times K$,
all continuous functions in $\mathcal{G}$ can be seen as restrictions of
continuous functions in $K\times K$. And viewing $C(K\times K)=C(K)\otimes
C(K)$, we have that every continuous function in $\mathcal{G}$ can be written as a
limit of sums of elements of the type $a\otimes b$ where $a,b\in C(K)$. We can
do this both with respect to the sup norm and to the Hilbert-module norm because
of proposition \ref{prop:BasicResultsCorr}.
\end{remark}

We also note that the code map $F:\Omega \rightarrow K$ defined in
proposition \ref{prop:CodeMap} induces an injection of $C(K)$ in $%
C(\Omega )$.

\begin{proposition}
If the IFS $\Gamma$ satisfies the finite branch condition or the open set condition, then $\mathcal{O}_{\Gamma }$ is the sub-C*-algebra of $\mathcal{O}_{d}$ generated
by $C(K)$ and $S$.
\end{proposition}

\begin{proof}
As a Cuntz-Pimsner algebra is generated by copies of elements of the algebra and copies of elements of the module, we have that $\mathcal{O}_{\Gamma }$ is
generated by all elements $a\in C(K)$ and $\xi\in C(\mathcal{G})$. It suffices
to note that $\psi (\boldsymbol{1})=S$ where $\boldsymbol{1}$ is the identity of $C(\mathcal{G})$ and $\psi (a\otimes b)=\iota (a)\psi (\boldsymbol{1})\iota (b)$
for $a,b\in C(K)$.
\end{proof}

We recall a definition from \cite{KW1} and give a different proof of the isomorphism
between $\mathcal{O}_{\gamma }$ and $\mathcal{O}_{d}$ for a certain class of
IFS.

\begin{definition}
We say that the IFS $\left\{ \gamma _{i}\right\} _{i=1}^{d}$ satisfies the
cograph separation condition if the cographs
\[
\mathcal{G}_i=\left\{ (\gamma _{i}(y),y)\in K\times K:y\in K\right\} 
\]%
are disjoint.
\end{definition}

We note that the cographs of an IFS are always closed subsets of $\mathcal{G}
$ and if the IFS satisfies the cograph separation condition then they are
also open. In this case, there are no branched points and in particular, it
satisfies the finite branch condition.

\begin{proposition}
If the IFS $\left\{ \gamma _{i}\right\} _{i=1}^{d}$ satisfies the cograph
separation condition then $\mathcal{O}_{\gamma }\simeq \mathcal{O}_{d}$.
\end{proposition}

\begin{proof}
If $\chi _{\mathcal{G}_i}$ is the characteristic function of $\mathcal{G}_i$ then
it belongs to $C(\mathcal{G})$ and we note that $\psi (\chi _{\mathcal{G}_i})=\chi_{\overline{i}}\ast S$ where $%
\chi_{\overline{i}}$ is the characteristic function of the cylinder $\overline{i}$.
As we've seen in subsection \ref{sec:CuntzAlg}, the elements $S_{i}=d^{1/2}\chi_{\overline{i}}\ast S=d^{1/2}\psi (\chi _{\mathcal{G}_i})$ are $d$ isometries that
satisfies the Cuntz relations and generates $\mathcal{O}_{d}$.
\end{proof}

\subsection{The case of inverse branches of a continuous function}

In this subsection we suppose that there exists a continuous function $\gamma:K\to K$ such that $\gamma\circ\gamma_i=id$ for all $i\in\{1,\ldots,d\}$. Our goal is to show that if the IFS satisfies the finite branch condition than we can see $\mathcal{O}_{\Gamma}$ as an Exel's crossed product by endomorphism \cite{E1}.

We note that $\gamma$ needs not to be a local homeomorphism and in this case we cannot use the construction by Renault \cite{R1} and Deaconu \cite{D1}. But when it does, their construction is isomorphic to Exel's one \cite{EV1}.

We begin by recalling the ingredients to build Exel's crossed product. Let $A$ be a unital C*-algebra and suppose we're given:
	\begin{itemize}
		\item An unital injective endomorphism $\alpha:A\to A$.
		\item A transfer operator $L:A\to A$ for $\alpha$, that is, a positive continuous linear map such that $L(\alpha(a)b)=aL(b)$ for $a,b\in A$. We suppose that $L(1)=1$.
	\end{itemize}

Let $\mathcal{T}(A,\alpha,L)$ be the universal C*-algebra generated by a copy of A and an element $\widehat{S}$ with relations:
	\begin{itemize}
		\item[(i)] $\widehat{S}a=\alpha(a)\widehat{S},$
		\item[(ii)] $\widehat{S}^{*}a\widehat{S}=L(a),$
	\end{itemize}
for $a\in A$. Note that the canonical map from $A$ to $\mathcal{T}(A,\alpha,L)$ is injective.

\begin{definition}
	A redundancy is a pair $(a,k)\in A\times\overline{A\widehat{S}\widehat{S}^{*}A}$ such that $ab\widehat{S}=kb\widehat{S}$ for all $b\in A$.
\end{definition}

\begin{definition}
	The Exel's crossed product $A\rtimes_{\alpha,L}\mathbb{N}$ is the quotient of $\mathcal{T}(A,\alpha,L)$ by the closed two-sided ideal generated by the set of differences $a-k$ for all redundancies $(a,k)$.
\end{definition}

In our case, let $A=C(K)$ and $\alpha:A\to A$ be given by
$$\alpha(a)=a\circ\gamma.$$
Then $L:A\to A$ defined by
$$L(a)=\frac{1}{d}\sum_{i=1}^d a\circ\gamma_i$$
is a transfer operator for $\alpha$.

\begin{theorem}
Let $\Gamma=\left\{ \gamma _{i}\right\} _{i=1}^{d}$ be an IFS satisfying the finite branch condition or the open set condition, and let $A,\alpha$ and $L$ be as above then $A\rtimes_{\alpha,L}\mathbb{N}$ is isomorphic to $\mathcal{O}_{\Gamma}$.
\end{theorem}

\begin{proof}
	The steps of the proof are similar to what we have done last subsection. Let $(A=C(K),E=C(\mathcal{G}))$ be the C*-correspondence given in last subsection. We start by giving a covariant representation of $(A,E)$ in $A\rtimes_{\alpha,L}\mathbb{N}$. 
	
	Let $\iota:A\to A\rtimes_{\alpha,L}\mathbb{N}$ be the canonical inclusion and $\psi:E\to A\rtimes_{\alpha,L}\mathbb{N}$ be given by
	$$\psi(a\otimes b)=aSb$$
for $a,b\in A$. To show that $\psi$ is well defined in all $C(\mathcal{G})$, we let $\sum_j a_j\otimes b_j$ be a finite sum where $a_j,b_j\in A$, then
	$$\left\|\psi\left(\sum_j a_j\otimes b_j\right)\right\|=\left\|\sum_j a_jS b_j\right\|=\left\|\sum_j a_j\alpha(b_j)S\right\|\leq\left\|\sum_j a_j\alpha(b_j)\right\|=$$
	$$=\sup_{x\in K}\left |\sum_j a_j(x)b_j(\gamma(x))\right |\leq\left(\sup_{x\in K}\sum_{i=1}^d{\left |\sum_j a_j(\gamma_i(x)) b_j(x)\right |^2}\right)^{1/2}=\left\|\sum_j a_j\otimes b_j\right\|_2$$
	where $\left\|\cdot\right\|_2$ is the norm in $C(\mathcal{G})$ thinking of $C(\mathcal{G})$ as an $A$-Hilbert module. To justify that the second inequality above holds, we note that because $K$ is self-similar, for any $x\in K$ there is $y\in K$ such that $x=\gamma_i(y)$ for some $i=1,\ldots,d$.
	
	We have to show that $(\iota,\psi)$ is a Cuntz-Pimsner covariant representation, i.e., it satisfies conditions (i)-(iii) of definition \ref{def:PimsnerCovRep}. Condition (i) is easily verified. For (ii), it suffices to show for monomials $a\otimes b, e\otimes f\in C(\mathcal{G})$ because of linearity and continuity. We have
	$$\left<a\otimes b,e\otimes f\right>(y)=\sum_i\overline{a(\gamma_i(y))b(y)}e(\gamma_i(y))f(y)=b^{*}(y)L(a^{*}e)(y)f(y)$$
	and then
	$$\iota(\left<a\otimes b,e\otimes f\right>)=b^{*}L(a^{*}e)f=b^{*}S^{*}a^{*}eSf=\psi(a\otimes b)^{*}\psi(e\otimes f).$$
	
	Finally, for condition (iii), we take $a\in J_E$ such that $\mathrm{supp}(a)\subseteq K\backslash B$ and $\xi_k,\eta_k$ as in remark \ref{rem:CompactInvIm}. Then
$$(\iota,\psi)^{(1)}(\phi(a))=\sum_k\psi(\xi_k)\psi(\eta_k)^{*}=\sum_ka\sqrt{\varphi_k}SS^{*}\sqrt{\varphi_k}$$
and we have to show that this equals $a$ inside $A\rtimes_{\alpha,L}\mathbb{N}$. For that, we show that the pair $(a,\sum a\sqrt{\varphi_k}\widehat{S}\widehat{S}^{*}\sqrt{\varphi_k})$ is a redundancy. We let $b\in A$, then
 $$\sum a\sqrt{\varphi_k}\widehat{S}\widehat{S}^{*}\sqrt{\varphi_k}b\widehat{S}=\sum a\sqrt{\varphi_k}(\alpha\circ L)(\sqrt{\varphi_k}b)\widehat{S}.$$
 To show that the pair above is a redundancy, it suffices to show that
 $$b(x)=\sum\sqrt{\varphi_k}(x)(\alpha\circ L)(\sqrt{\varphi_k}b)(x)$$
 for $x\in\mathrm{supp}(a)$. For such $x$, we have that $x\notin B$ and hence, there is a unique $i_0$ and a unique $y$ such that $\gamma_{i_0}(y)=x$. If $\varphi_k(x)=\varphi_k(\gamma_{i_0}(y))\neq 0$ then $i_0\in I(x_k)$ because of (iii) of lemma \ref{lemma:NiceOpenSet} and because of (ii) we have that $\gamma_i(y)\notin U_{x_k}$ for $i\neq i_0$. It follows that
	$$\sum\sqrt{\varphi_k}(x)(\alpha\circ L)(\sqrt{\varphi_k}b)(x)=\sum\sqrt{\varphi_k}(x)\sum_{i=1}^d\sqrt{\varphi_k}(\gamma_i(\gamma(x)))b(\gamma_i(\gamma(x)))=$$	
$$\sum\sqrt{\varphi_k}(x)\sum_{i=1}^d\sqrt{\varphi_k}(\gamma_i(y))b(\gamma_i(y))=\sum\sqrt{\varphi_k}(x)\sqrt{\varphi_k}(\gamma_{i_0}(y))b(\gamma_{i_0}(y))=$$
	$$\sum\varphi_k(x)b(x)=b(x).$$
	
	By the universality of $\mathcal{O}_{\Gamma}$, we have a homomorphism $\iota\times\psi:\mathcal{O}_{\Gamma}\to A\rtimes_{\alpha,L}\mathbb{N}$. Since $A\rtimes_{\alpha,L}\mathbb{N}$ is generated by $A$ and $S$, and $\iota(A)=A$, $\psi(\boldsymbol{1})=S$, where $\boldsymbol{1}$ is the unity of $C(\mathcal{G})$, we have that $\iota\times\psi$ is surjective.
	
	To show that $\iota\times\psi$ is injective, we first note that $\iota$ is faithful \cite{E2}. Then we see that $\beta:\mathbb{T}\to A\rtimes_{\alpha,L}\mathbb{N}$ given by $\beta_z(a)=a$ and $\beta_z(S)=zS$ is an action of the circle in $A\rtimes_{\alpha,L}\mathbb{N}$ \cite{E2} which clearly satisfies the conditions of lemma \ref{lemma:GaugeInvariance}.
\end{proof}

\subsubsection*{Acknowledgements} The author would like to thank his three advisors: Ruy Exel, Artur Lopes and Jean Renault. The author would also like to thank the Université d'Orléans and the MAPMO for their hospitality.

\bibliographystyle{amsplain}

\end{document}